\documentclass[11pt]{article}
\usepackage{amsmath,amssymb,amsthm}
\usepackage{epsfig}
\usepackage{color}
\usepackage{enumerate}
\usepackage{algorithm}
\usepackage{algorithmic}
\usepackage{hyperref}
\usepackage{enumerate}
\usepackage{graphicx}
\usepackage{tikz}
\usepackage{tkz-graph}
\usepackage{tkz-berge}
\usetikzlibrary{arrows.meta,shapes.arrows}
\hypersetup{colorlinks=true, linkcolor=blue, citecolor=blue,
urlcolor=blue}

\hypersetup{colorlinks=true}

\hypersetup{colorlinks=true, linkcolor=blue, citecolor=blue,urlcolor=blue}

\title{Double domination and total $2$-domination in digraphs and their dual problems}

\author{Doost Ali Mojdeh$^1$ and Babak Samadi$^2$\\[0.5cm]
Department of Mathematics, University of Mazandaran,\\ Babolsar, Iran$^{1,2}$\\
{\it damojdeh@umz.ac.ir$^1$, } {\it samadibabak62@gmail.com$^2$}\\[0.2cm]
}
\date{}

\newtheorem{theorem}{Theorem}[section]
\newtheorem{corollary}[theorem]{Corollary}

\newtheorem{proposition}[theorem]{Proposition}

\theoremstyle{definition}
\newtheorem{definition}[theorem]{Definition}
\theoremstyle{remark}

\begin{document}

\maketitle

\begin{abstract}

\noindent \ \ \
A subset $S$ of vertices of a digraph $D$ is a double dominating set (total $2$-dominating set) if every vertex not in $S$ is adjacent from at least two vertices in $S$, and every vertex in $S$ is adjacent from at least one vertex in $S$ (the subdigraph induced by $S$ has no isolated vertices). The double domination number (total $2$-domination number) of a digraph $D$ is the minimum cardinality of a double dominating set (total $2$-dominating set) in $D$. In this work, we investigate these concepts which can be considered as two extensions of double domination in graphs to digraphs, along with the concepts $2$-limited packing and total $2$-limited packing which have close relationships with the above-mentioned concepts.
\end{abstract}
\vspace{1mm}


\section{Introduction and preliminaries}

Throughout this paper, we consider $D=(V(D),A(D))$ as a finite digraph with vertex set $V=V(D)$ and arc set $A=A(D)$ with neither loops nor multiple arcs (although pairs of opposite arcs are allowed). Also, $G=(V(G),E(G))$ stands for a simple finite graph with vertex set $V(G)$ and edge set $E(G)$. We use \cite{bg} and \cite{we} as references for basic terminology and notation in digraphs and graphs, respectively, which are not defined here.

For any two vertices $u,v\in V(D)$, we write $(u,v)$ as the \emph{arc} with direction from $u$ to $v$, and say $u$ is \emph{adjacent to} $v$, or $v$ is \emph{adjacent from} $u$. We also say $u$ and $v$ are {\em adjacent with} each other. Given a subset $S$ of vertices of $D$ and a vertex $v\in V(D)$, the {\em in-neighborhood} of $v$ from $S$ ({\em out-neighborhood} of $v$ to $S$) is $N_S^{-}(v)=\{u\in S\mid(u,v)\in A(D)\}$ ($N_S^{+}(v)=\{u\in S\mid(v,u)\in A(D)\}$). The \emph{in-degree} of $v$ from $S$ is $deg_S^-(v)=|N_S^{-}(v)|$ and the \emph{out-degree} of $v$ to $S$ is $deg_S^+(v)=|N_S^{+}(v)|$. Moreover, $N_S^{-}[v]=N_{S}^{-}(v)\cup\{v\}$ ($N_S^{+}[v]=N_{S}^{+}(v)\cup\{v\}$) is the {\em closed in-neighborhood} ({\em closed out-neighborhood}) of $v$ from (to) $S$. In particular, if $S=V(D)$, then we simply say (closed) (in or out)-neighborhood and (in or out)-degree, and write $N_D^{-}(v)$, $N_D^{+}(v)$, $N_D^{-}[v]$, $N_D^{+}[v]$, $deg_D^-(v)$ and $deg_D^+(v)$ instead of $N_{V(D)}^{-}(v)$, $N_{V(D)}^{+}(v)$, $N_{V(D)}^{-}[v]$, $N_{V(D)}^{+}[v]$, $deg_{V(D)}^-(v)$ and $deg_{V(D)}^+(v)$, respectively (we moreover remove the subscripts $D$, $V(D)$ if there is no ambiguity with respect to the digraph $D$). For a graph $G$, $\Delta=\Delta(G)$ and $\delta=\delta(G)$ represent the maximum and minimum degrees of $G$. In addition, for a digraph $D$, ($\Delta^{+}=\Delta^+(D)$ and $\delta^{+}=\delta^+(D)$) $\Delta^{-}=\Delta^-(D)$ and $\delta^{-}=\delta^-(D)$ represent the maximum and minimum (out-degrees) in-degrees of $D$. Given two subsets $A$ and $B$ of vertices of $D$, by $(A,B)_D$ we mean the sets of arcs of $D$ going from $A$ to $B$. Finally, we let $N[v]=N^{-}[v]\cup N^{+}[v]$ for each vertex $v$ of a digraph $D$.

We denote the {\em converse} of a digraph $D$ by $D^{-1}$, obtained by reversing the direction of every arc of $D$. A vertex $v\in V(D)$ with $deg^{+}(v)+deg^{-}(v)=1$ is called an {\em end-vertex}. A {\em penultimate vertex} is a vertex adjacent with an end-vertex. A digraph $D$ is {\em connected} if its underlying graph is connected. A {\it directed tree} is a digraph in which its underlying graph is a tree. A digraph $D$ is said to be {\em functional} ({\em contrafunctional}) if every vertex of $D$ has out-degree (in-degree) one. For more information about this subject the reader can consult \cite{h} (\cite{hnc}).

A vertex $v\in V(D)$ ($v\in V(G)$) is said to dominate itself and its out-neighbors (neighbors). A subset $S\subseteq V(D)$ ($S\subseteq V(G)$) is a {\em dominating set} in $D$ ($G$) if all vertices are dominated by the vertices in $S$. The {\em domination number} $\gamma(D)$ ($\gamma(G)$) is the minimum cardinality of a dominating set in the digraph $D$ (graph $G$). The concept of domination in digraphs was introduced by Fu \cite{fu} and has been extensively studied in several papers including \cite{chy,hq,l1,msy}, while the very well-known same topic in graphs was intdoduced by Berge \cite{b} and Ore \cite{O}. The reader is referred to \cite{hhs2} for more details on this topic. Ouldrabah et al. \cite{obb} defined the concept of $k$-domination in digraphs, as a transformation of the same topic in graphs (see \cite{fj1} and \cite{fj2}), as follows. A subset $S$ of vertices in a digraph $D$ is a $k$-dominating set if $|N^{-}(v)\cap S|\geq k$ for every vertex $v$ in $V(D)\setminus S$. The $k$-domination number of $D$, denoted by $\gamma_{k}(D)$, is the minimum cardinality of a $k$-dominating set in $D$. Clearly, this concept is a generalization of the concept of domination in digraphs.

As digraphs are extensions graphs (note that a graph can be considered as a symmetric digraph), we can expect that a well-known concept in graph theory can be extended to digraph theory in different ways. For example, Arumugam et al. \cite{ajv} investigated two extensions of the {\em total domination} (in graphs) to digraphs in two different ways, namely, {\em open domination} and {\em total domination} in digraphs.

The {\em $k$-tuple domination number} $\gamma_{\times k}(G)$ of a graph $G$ with $\delta(G)\geq k-1$ is the minimum cardinality of a subset $S\subseteq V(G)$ such that $|N[v]\cap S|\geq k$, for each vertex $v\in V(G)$. In particular, the $2$-tuple domination number is called a {\em double domination number}. This concept was first introduced by Harary and Haynes in \cite{hh}. Gallant et al. \cite{gghr} introduced the concept of limited packing in graphs as follows. The {\em $k$-limited packing number} $L_{k}(G)$ of a graph $G$ is the maximum cardinality of a subset $B\subseteq V(G)$ such that $|N[v]\cap B|\leq k$, for each vertex $v\in V(G)$. Note that $L_{1}(G)=\rho(G)$ is the well-known {\em packing number} in the graph $G$.

The concept of double domination in graphs can be extended to digraphs in two different ways.

\begin{definition}\label{def1}
Let $D$ be a digraph with $\delta^{-}(D)\geq1$. A subset $S\subseteq V(D)$ is a {\em double dominating set} in $D$ if every vertex is dominated by at least two vertices in $S$. The {\em double domination number} $\gamma_{\times2}(D)$ is the minimum cardinality of a double dominating set in $D$.
\end{definition}

\begin{definition}\label{def2}
Let $D$ be a digraph with no isolated vertices. A subset $S\subseteq V(D)$ is a {\em total $2$-dominating set} in $D$ if $D\langle S\rangle$ has no isolated vertices and every vertex in $V(D)\setminus S$ is dominated by at least two vertices in $S$. The {\em total $2$-domination number} $\gamma^{t}_{\times2}(D)$ is the minimum cardinality of a total $2$-dominating set in $D$.
\end{definition}

We remark that the definition of total $2$-dominating sets is more general than that of double dominating sets, as every double dominating set is a total $2$-dominating set (we have $\gamma(D)\leq \gamma^{t}_{\times2}(D)\leq \gamma_{\times2}(D)$ for all digraphs $D$ with $\delta^{-}(D)\geq1$). In fact, Definition \ref{def1} is more restricted than Definition \ref{def2} so that the first one cannot be used for some important families of digraphs like acyclic digraphs (the digraphs with no directed cycle), especially directed trees.

Regarding the $2$-limited packing in graphs, we can extend this concept to digraphs in two different ways.

\begin{definition}\label{def3}
A subset $B\subseteq V(D)$ is a {\em $2$-limited packing} in the digraph $D$ if $|N^{+}[v]\cap B|\leq2$, for every vertex $v\in V(D)$. The {\em $2$-limited packing number} $L_{2}(D)$ is the maximum cardinality of a $2$-limited packing in $D$.
\end{definition}

\begin{definition}\label{def4}
A subset $B\subseteq V(D)$ is a {\em total $2$-limited packing} in the digraph $D$ if every vertex in $B$ is adjacent with at most one vertex in $B$ and every vertex in $V(D)\setminus B$ is adjacent to at most two vertices in $B$. The {\em total $2$-limited packing number} $L^{t}_{2}(D)$ is the maximum cardinality of a total $2$-limited packing in $D$.
\end{definition}
Note that $L_{1}(D)=\rho(D)$ is the {\em packing number} of the digraph $D$. Comparing the last two definitions we can readily observe that $\rho(D)\leq L_{2}^{t}(D)\leq L_{2}(D)$, for all digraphs $D$.

Note that, for various reasons, the small values of $k$ (especially $k\in\{1,2\}$) regarding the above-mentioned parameters  have attracted more attention from the experts in domination theory rather than the large ones. One reason is that for the large values of $k$ we lose some important families of graphs (for example, the $k$-tuple domination number cannot be studied for trees when $k\geq3$), or we deal with a trivial problem (for example, for every graph $G$ with $k>\Delta(G)$, we have $L_{k}(G)=|V(G)|$). Another reason is that many results for the case $k\in\{1,2\}$ can be easily generalized to the general case $k$. Moreover, one may obtain stronger results for the small values of $k$ rather than the large ones. For more evidences on these pieces of information the reader can be referred to \cite{bbg}, \cite{cg} and \cite{hh}.

In this paper, we initiate the investigation of the parameters given in Definitions \ref{def1}-\ref{def4}. We derive their computational complexity and give several lower and upper bound on these parameters. We show that the problems given in Definition \ref{def1} (Definition \ref{def2}) and Definition \ref{def3} (Definition \ref{def4}) are the dual problems (on the instances of directed trees). In dealing with the total $2$-domination number and the total $2$-limited packing number our main emphasis is on directed trees, by which we prove that  $L^{t}_{2}(D)+L^{t}_{2}(D^{-1})$ can be bounded from above by $16n/9$ for any digraph $D$ of order $n$. Also, we bound the total $2$-domination number of a directed tree from below and characterize the directed trees attaining the bound.


\section{Computational complexity}

We first consider the following two well-known decision problems in domination theory.

\begin{equation}\label{comp1}
\begin{tabular}{|l|}
  \hline
  \mbox{DOMINATING SET PROBLEM}\\
  \mbox{INSTANCE: A graph $G$ and a positive integer $k$.}\\
  \mbox{QUESTION: Is $\gamma(G)\leq k$?}\\
  \hline
\end{tabular}
\end{equation}

\begin{equation}\label{comp2}
\begin{tabular}{|l|}
  \hline
  \mbox{PACKING PROBLEM}\\
  \mbox{INSTANCE: A graph $G$ and a positive integer $k$.}\\
  \mbox{QUESTION: Is $\rho(G)\geq k$?}\\
  \hline
\end{tabular}
\end{equation}

We make use of these two problems which are known to be NP-complete from \cite{hhs2} and \cite{packing}, respectively, in order to study the complexity of the problems introduced in this paper. Indeed. we deal with the following decision problems.

\begin{equation}\label{comp3}
\begin{tabular}{|l|}
  \hline
  \mbox{(TOTAL $2$-) DOUBLE DOMINATING SET PROBLEM}\\
  \mbox{INSTANCE: A digraph $D$ (with minimum in-degree at least one) with no isolated}\\ \mbox{vertices, and a positive integer $k$.}\\
  \mbox{QUESTION: Is ($\gamma^{t}_{\times2}(D)\leq k$?) $\gamma_{\times2}(D)\leq k$?}\\
  \hline
\end{tabular}
\end{equation}

\begin{equation}\label{comp4}
\begin{tabular}{|l|}
  \hline
  \mbox{(TOTAL $2$-) $2$-LIMITED PACKING PROBLEM}\\
  \mbox{INSTANCE: A digraph $D$ and a positive integer $k$.}\\
  \mbox{QUESTION: Is ($L^{t}_{2}(D)\geq k$?) $L_{2}(D)\geq k$?}\\
  \hline
\end{tabular}
\end{equation}

We next present NP-completeness results for the digraph problems listed above. Recall first that for a graph $G$, the {\em complete biorientation} $\overleftrightarrow{G}$ of $G$ is a digraph $D$ obtained from $G$ by replacing each edge $xy\in E(G)$ by the pair of arcs $(x,y)$ and $(y,x)$.

\begin{theorem}\label{complexity}
The problems given in the rectangles \emph{(\ref{comp3})} and \emph{(\ref{comp4})} are NP-complete.
\end{theorem}
\begin{proof}
The problems are clearly in NP since checking that a given set is indeed a double dominating set, total $2$-dominating set, $2$-limited packing, or total $2$-limited packing can be done in polynomial time.

Let $G$ be a graph. We then consider the complete biorientation $\overleftrightarrow{G}$ of $G$. It is easy to check that a set $S\subseteq V(G)$ is a dominating set in $G$ if and only if $S\subseteq V(\overleftrightarrow{G})$ is a dominating set in $\overleftrightarrow{G}$. This shows that $\gamma(G)=\gamma(\overleftrightarrow{G})$. We moreover have $\rho(G)=\rho(\overleftrightarrow{G})$, by a similar fashion. We now deduce from the problems (\ref{comp1}) and (\ref{comp2}) that the corresponding problems in digraphs are NP-complete.

We begin with a digraph $D$ with $V(D)=\{v_{1},\cdots,v_{n}\}$. For each $1\leq i\leq n$, we add new vertices $w_{i}$ and $u_{i}$, and arcs $(w_{i},u_{i})$, $(u_{i},w_{i})$ and $(u_{i},v_{i})$. We denote the obtained digraph by $D'$. Note that every double dominating set and total $2$-dominating set $S'$ in $D'$ contains both $w_{i}$ and $u_{i}$, for each $1\leq i\leq n$. Moreover, $|S'\cap V(D)|$ must be at least as large as $\gamma(D)$ so as to the vertices of $D$ can be double dominated (total $2$-dominated) by $S'$. On the other hand, for each $\gamma(D)$-set $S$, $S\cup(\cup_{i=1}^{n}\{w_{i},u_{i}\})$ is both a double dominating set and a total $2$-dominating set in $D'$. The above argument shows that $\gamma_{\times2}(D')=\gamma^{t}_{\times2}(D')=2n+\gamma(D)$. Now by taking $j=k+2n$, we have $\gamma_{\times2}(D')\leq j$ ($\gamma^{t}_{\times2}(D')\leq j$) if and only if $\gamma(D)\leq k$. So, the problems given in (\ref{comp3}) are NP-complete.

We now construct the digraph $D''$ from $D$ by adding a new vertex $x_{i}$ and a new arc $(v_{i},x_{i})$, for each $1\leq i\leq n$. It is easy to check that $B\cup \{x_{i}\}_{i=1}^{n}$ is both a $2$-limited packing and a total $2$-limited packing in $D''$, in which $B$ is a $\rho(D)$-set. Therefore, $L_{2}(D''),L^{t}_{2}(D'')\geq n+\rho(D)$.

Let $B''$ be a $L_{2}(D'')$-set ($L^{t}_{2}(D'')$-set). Let $u_{i}\notin B''$, for some $1\leq i\leq n$. If $v_{i}\notin B''$, then it must be adjacent to precisely two vertices in $B''$, for otherwise $B''\cup\{u_{i}\}$ would be a $2$-limited packing (total $2$-limited packing) in $D''$ which contradicts the maximality of $B''$. Then $(B''\setminus\{w_{i}\})\cup\{u_{i}\}$ is a $L_{2}(D'')$-set ($L^{t}_{2}(D'')$-set) containing $u_{i}$, in which $w_{i}\in N^{+}(v_{i})\cap B''$. Now if $v_{i}\in B''$, we easily observe that $(B''\setminus\{v_{i}\})\cup\{u_{i}\}$ is such a $L_{2}(D'')$-set ($L^{t}_{2}(D'')$-set). Therefore, we may assume that $u_{i}\in B''$ for all $1\leq i\leq n$.

We then note that $|B''\cap V(D'')|$ must be less than or equal to $|B|$. If this is not true, then it is not hard to see that $B''$ is neither a $2$-limited packing nor a total $2$-limited packing in $D''$, a contradiction. Therefore, $L_{2}(D''),L^{t}_{2}(D'')\leq n+\rho(D)$. It now follows that $L_{2}(D'')=L^{t}_{2}(D'')=n+\rho(D)$. Now by taking $j=k+n$, we have $L_{2}(D'')\geq j$ ($L^{t}_{2}(D'')\geq j$) if and only if $\rho(D)\geq k$. So, the problems given in (\ref{comp4}) are NP-complete. This completes the proof.
\end{proof}

As a consequence of the result above, we conclude that the problems of computing the parameters given in Definitions \ref{def1}-\ref{def4} are NP-hard. Taking into account this fact, it is desirable to bound their values with respect to several different invariants of the digraph. In the next two sections we exhibit such results.


\section{Bounding $\gamma_{\times2}(D)$ and $L_{2}(D)$ for general digraphs}

In this section, we discuss some results about the digraph parameters $L_{k}(D)$ and $\gamma_{\times r}(D)$ for $k,r\in\{1,2\}$.

\begin{proposition}\label{T1}
Let $D$ be a digraph of order $n$. Then,\vspace{1mm}

\emph{(i)} If $\Delta^{+}\geq1$, then $L_{2}(D)\geq \rho(D)+1$.\vspace{1mm}

\emph{(ii)} If $\delta^{-}\geq1$, then $\gamma_{\times2}(D)\geq \gamma(D)+1$.\vspace{1mm}\\
These bounds are sharp.
\end{proposition}
\begin{proof}
(i) Let $B$ be a maximum packing in $D$. Since $|N^{+}[v]\cap B|\leq1$ for every vertex $v\in V(D)$ and $\Delta^{+}\geq1$, it follows that at least one vertex in the closed out-neighborhood of a vertex of maximum out-degree belongs to $V(D)\setminus B$. So, $B\neq V(D)$. Let $u\in V(D)\setminus B$. It is easily observed that $B\cup\{u\}$ is a $2$-limited packing in $D$. Therefore, $L_{2}(D)\geq|B\cup\{u\}|=\rho(D)+1$.

(ii) Let $S$ be a minimum double dominating set in $D$. Let $u\in S$. It is easy to see that $S\setminus\{u\}$ is a dominating set in $D$. So, $\gamma_{\times2}(D)\leq|S|-1=\gamma(D)-1$.

The complete biorientation $\overleftrightarrow{K_{n}}$ of the complete graph $K_{n}$ of order $n\geq2$ shows that the bounds in both (i) and (ii) are sharp.
\end{proof}

\begin{proposition}\label{T2}
Let $k,r\in\{1,2\}$ and let $D$ be a digraph with $\delta^{-}(D)\geq r-1$. Then, $L_{k}(D)\leq \frac{k}{r}\gamma_{\times t}(D)$. This bound is sharp.
\end{proposition}
\begin{proof}
Let $B$ and $S$ be a $L_{k}(D)$-set and a $\gamma_{\times r}(D)$-set, respectively. We define $Z=\{(u,v)\mid u\in B,\ v\in S\ \mbox{and}\ u\in N^{+}[v]\}$. Here $(u,v)$ is an ordered pair of $u$ and $v$. Since $B$ is a $k$-limited packing in $D$, at most $k$ vertices in $B$ belong to the closed out-neighborhood of every vertex in $S$. Therefore, $|Z|\leq k|S|$. Moreover, since $S$ is an $r$-tuple dominating set in $D$, every vertex in $B$ belongs to the closed out-neighborhood of at least $r$ vertices in $S$. Therefore, $r|B|\leq|Z|$. The desired upper bound follows now from the last two inequalities.

That the bound is sharp can be seen by considering the digraph $\overleftrightarrow{K_{n}}$ on $n\geq r$ vertices.
\end{proof}

In the next result, we bound the $2$-limited packing number of a digraph from above just in terms of its order and minimum in-degree. We first introduce a family of digraphs in order to characterize all digraphs attaining the upper bound. Let $D'$ be a functional digraph with $V(D')=\{v_{1},\cdots,v_{n'}\}$. Choose $r\geq \Delta^{-}(D')$ such that $p=(r-1)n'\equiv0$ (mod $2$). Add a set of new vertices $U=\{u_{1},\cdots,u_{p/2}\}$ and new arcs $(u_{i},v_{j})$ such that

($i$) every vertex $u_{i}$ is incident with precisely two such arcs, and

($ii$) $deg^{-}(v_{j})=r$ for all $1\leq j\leq n'$.\\
Now add some arcs among the vertices $u_{i}$ and some arcs from $V(D')$ to $U$, such that $r$ is the minimum in-degree of the constructed digraph. Let $\Omega$ be the family of digraphs $D$ constructed as above.

\begin{theorem}\label{123}
Let $D$ be a digraph of order $n$ with minimum in-degree $\delta^{-}\geq1$. Then, $L_{2}(D)\leq \frac{2n}{\delta^{-}+1}$ with equality if and only if $D\in \Omega$.
\end{theorem}
\begin{proof}
Let $B$ be a $L_{2}(D)$-set. By the definition, every vertex in $V(D)\setminus B$ has at most two out-neighbors in $B$. Thus,
\begin{equation}\label{Bob1}
|(V(D)\setminus B,B)_{D}|\leq2(n-|B|).
\end{equation}

On the other hand, since every vertex in $B$ is adjacent to at most one vertex in $B$ and $\sum_{v\in B}deg^{-}_{B}(v)=\sum_{v\in B}deg^{+}_{B}(v)$, we have
\begin{equation}\label{Bob2}
\begin{array}{lcl}
|B|(\delta^{-}-1)\leq \sum_{v\in B}deg^{-}(v)-\sum_{v\in B}deg^{+}_{B}(v)&=&\sum_{v\in B}(deg^{-}(v)-deg^{-}_{B}(v))\\
&=&\sum_{v\in B}deg^{-}_{V(D)\setminus B}(v)=|(V(D)\setminus B,B)_{D}|.
\end{array}
\end{equation}
Together inequalities (\ref{Bob1}) and (\ref{Bob2}) imply the desired upper bound.

Suppose that $D\in \Omega$. It is easily seen that $V(D')$ is a $2$-limited packing in $D$. Moreover, $\delta^{-}(D)=r$ and $n=n'+p/2$. Therefore, $|V(D')|=n'=2n/(\delta^{-}+1)$. We now have $L_{2}(D)\geq2n/(\delta^{-}+1)$, implying the desired equality.

Let the upper bound hold with equality. Then both (\ref{Bob1}) and (\ref{Bob2}) hold with equality, necessarily. The equality in (\ref{Bob2}) shows that every vertex in $B$ is adjacent to exactly one vertex in $B$. Therefore, $D\langle B\rangle$ is a functional digraph. Also, all vertices of this digraph have in-degree $\delta^{-}$. The equality in (\ref{Bob1}) shows that every vertex in $V(D)\setminus B$ is adjacent to exactly two vertives in $B$. Thus, the membership $D$ in $\Omega$ easily follows by choosing $D\langle B\rangle$, $\delta^{-}$ and $V(D)\setminus B$ for $D'$, $r$ and $U$, respectively, in the description of $\Omega$.
\end{proof}

We conclude this section by bounding the total $2$-domination number of a digraph from below in terms of its order and maximum out-degree. Indeed, the following theorem for total $2$-domination can be considered as a result analogous to Theorem \ref{123} for total $2$-limited packing. Similarly to that for Theorem \ref{123}, we introduce a family of digraphs so as to characterize all digraphs attaining the lower bound given in the next theorem. We begin with a contrafunctional digraph $D'$ with $V(D')=\{v_{1},\cdots,v_{n'}\}$. Choose $r\geq \Delta^{+}(D')$ such that $q=(r-1)n'\equiv0$ (mod $2$). Add a set of new vertices $U=\{u_{1},\cdots,u_{q/2}\}$ and new arcs $(v_{i},u_{j})$ such that

($i$) every vertex $u_{j}$ is incident with precisely two such arcs, and

($ii$) $deg^{+}(v_{i})=r$ for all $1\leq i\leq n'$.\\
Now add some arcs among the vertices $u_{j}$ and some arcs from $U$ to $V(D')$, such that $r$ is the maximum out-degree of the constructed digraph. Let $\Theta$ be the family of digraphs $D$ constructed as above.

\begin{theorem}
For any digraph $D$ of order $n$ with maximum out-degree $\Delta^{+}$ and minimum in-degree $\delta^{-}\geq1$, $\gamma_{\times2}(D)\geq \frac{2n}{\Delta^{+}+1}$. Furthermore, the equality holds if and only if $D\in \Theta$.
\end{theorem}
\begin{proof}
Let $S$ be a $\gamma_{\times2}(D)$-set. Every vertex in $V(D)\setminus S$ is adjacent from at least two vertices in $S$, by the definition. Hence,
\begin{equation}\label{Bob3}
2(n-|S|)\leq|(S,V(D)\setminus S)_{D}|.
\end{equation}

On the other hand, $\sum_{v\in S}deg^{-}_{S}(v)=\sum_{v\in S}deg^{+}_{S}(v)$ and every vertex in $S$ is adjacent from at least one vertex in $S$. Therefore,
\begin{equation}\label{Bob4}
\begin{array}{lcl}
|(S,V(D)\setminus S)_{D}|=\sum_{v\in S}deg_{V(D)\setminus S}^{+}(v)&=&\sum_{v\in S}(deg^{+}(v)-deg_{S}^{+}(v))\\
&=&\sum_{v\in S}deg^{+}(v)-\sum_{v\in S}deg_{S}^{-}(v)\leq(\Delta^{+}-1)|S|.
\end{array}
\end{equation}
The desired lower bound now follows by the inequalities (\ref{Bob3}) and (\ref{Bob4}).

Suppose that $D\in \Theta$. Clearly, $r=\Delta^{+}(D)$ and $n=n'+q/2$. Moreover, $V(D')$ is a double dominating set in $D$. Since $|V(D')|=n'=2n/(\Delta^{+}+1)$, it follows that $\gamma_{\times2}(D)\geq2n/(\Delta^{+}+1)$, implying the equality in the lower bound.

Conversely, let the equality hold in the lower bound. Then both (\ref{Bob3}) and (\ref{Bob4}) hold with equality, necessarily. Therefore, every vertex in $S$ is adjacent from exactly one vertex in $S$. This shows that $D\langle S\rangle$ is a contrafunctional digraph. Also, the equality $\sum_{v\in S}deg^{+}(v)=\Delta^{+}|S|$ implies that every vertex of $D\langle S\rangle$ has the out-degree $\Delta^{+}$. On the other hand, the equality in (\ref{Bob3}) shows that every vertex in $V(D)\setminus S$ is adjacent from exactly two vertices in $S$. That the digraph $D$ is in $\Omega$ can be easily seen by choosing $D\langle S\rangle$, $\Delta^{+}$ and $V(D)\setminus S$ for $D'$, $r$ and $U$, respectively, in the description of $\Theta$.
\end{proof}


\section{Directed trees}

We first recall that a maximization problem \textbf{M} and a minimization problem \textbf{N}, defined on the same instances (such as graphs or digraphs), are {\em dual problems} if the value of every candidate solution $M$ to \textbf{M} is less than or equal to the value of every candidate solution $N$ to \textbf{N}. Often the ``value'' is cardinality. Analogously to many well known pairs of dual (graph or digraph) problems like {\em matching} and {\em vertex covering}, {\em packing} and {\em domination}, etc. the following theorem shows that the problems ``total $2$-domination'' and ``total $2$-limited packing'', on the instances of directed trees, are dual problems.

Recall that in a tree a support vertex is called a \textit{weak} (\textit{strong}) \textit{support vertex} if it is adjacent to (more than) one leaf. Also, a {\em double star} $S_{a,b}$ is a tree with exactly two non-leaf vertices in which one support vertex is adjacent to $a$ leaves and the other to $b$ leaves.

\begin{theorem}\label{imperative}
For any directed tree $T$ of order $n\geq2$, $L_{2}^{t}(T)\leq \gamma_{\times2}^{t}(T)$.
\end{theorem}
\begin{proof}
We proceed by induction on the order $n$. The result is obvious for $n=2$. Let $\tilde{T}$ be the underlying tree of $T$. It is easy to check that the result is true when $diam(\tilde{T})\leq3$. In such a case, we have $L_{2}^{t}(T)\leq \gamma_{\times2}^{t}(T)=n$. So, we may assume that $diam(\tilde{T})\geq4$. This implies that $n\geq5$. Assume that the inequality holds for all directed trees $T'$ of order $3\leq n'<n$. Let $T$ be a directed tree of order $n\geq4$. Suppose that $r$ and $x$ are two leaves of $\tilde{T}$ with $d(r,x)=diam(\tilde{T})$. Let $x$ be adjacent with $y$. Note that the choice of $x$ shows that all children of $y$ in $\tilde{T}$ are leaves. Let $B$ and $S$ be a $L_{2}^{t}(T)$-set and a $\gamma_{\times2}^{t}(T)$-set in $T$, respectively. Note that all end-vertices and penultimate vertices belong to every total $2$-dominating set in $T$. Suppose that $y$ is a strong support vertex in $\tilde{T}$. Then $S\setminus\{x\}$ is a total $2$-dominating set in $T'=T-x$. Moreover, it is easy to see that $B\setminus\{x\}$ is a total $2$-limited packing in $T'$. Using now the induction hypothesis we have
$$L_{2}^{t}(T)-1\leq|B\setminus\{x\}|\leq L_{2}^{t}(T')\leq \gamma_{\times2}^{t}(T')\leq|S|-1=\gamma_{\times2}^{t}(T)-1.$$

From now on, we assume that $y$ is a weak support vertex of $\tilde{T}$. Hence, $x$ is a sink or source of $T$ and $y$ is the unique vertex adjacent with $x$. We consider two cases depending on the behavior of $x$.

\textit{Case 1.} Suppose that $S\setminus\{x\}$ is a total $2$-dominating set in the directed tree $T'=T-x$. Moreover, $B\setminus\{x\}$ is a total $2$-limited packing in $T'$. Then, by using the induction hypothesis, we have again
$$L_{2}^{t}(T)-1\leq|B\setminus\{x\}|\leq L_{2}^{t}(T')\leq \gamma_{\times2}^{t}(T')\leq|S|-1=\gamma_{\times2}^{t}(T)-1.$$

\textit{Case 2.} Suppose that $S\setminus\{x\}$ is not a total $2$-dominating set in $T'$. This shows that $y$ is not adjacent with any vertex in $S\setminus\{x\}$. Here we need to consider two more possibilities.

\textit{Subcase 2.1.} Let $S''=S\setminus\{x,y\}$ be a total $2$-dominating set in $T''=T-x-y$. On the other hand, it is clear that $B''=B\setminus\{x,y\}$ is a total $2$-limited packing in $T''$. We then have
$$L_{2}^{t}(T)-2\leq|B''|\leq L_{2}^{t}(T'')\leq \gamma_{\times2}^{t}(T'')\leq|S''|=\gamma_{\times2}^{t}(T)-2.$$

\textit{Subcase 2.2.} Suppose now that $S''=S\setminus\{x,y\}$ is not a total $2$-dominating set in $T''$. This assumption along with the fact that $y$ is not adjacent with any vertex in $S\setminus\{x\}$ imply that there exists a vertex $z\in V(T'')\setminus S''$ such that $|N^{-}_{T''}(z)\cap S''|\leq1$. This shows that $|N^{-}_{T}(z)\cap S|=2$ and $y\in N^{-}(z)$, necessarily. Note that by our choice of $x$, all children of $z$ in $\tilde{T}$ are leaves or support vertices. If $z$ is adjacent with an end-vertex, then we have contradiction to the fact that $y$ is not adjacent with any vertex in $S\setminus\{x\}$. Therefore, all children of $z$ in $\tilde{T}$ are support vertices. Let $\tilde{T}_{z}$ be the subtree of $\tilde{T}$ rooted at $z$ consisting of $z$ and its descendants in $\tilde{T}$. Now consider the directed tree $T'''=T-V(\tilde{T}_{z})$ (our choice of $x$ and $diam(\tilde{T})\geq4$ imply that $|V(T''')|\geq2$). Let $z$ have $k$ children in $\tilde{T}$. It is easy to see that $S'''=S\setminus V(\tilde{T}_{z})$ is a total $2$-dominating set in $T'''$ with $|S'''|=\gamma_{\times2}^{t}(T)-2k$. On the other hand, $B'''=B\setminus V(\tilde{T}_{z})$ is a total $2$-limited packing in $T'''$ with $|B'''|\geq L_{2}^{t}(T)-2k$. Therefore,
$$L_{2}^{t}(T)-2k\leq|B'''|\leq L_{2}^{t}(T''')\leq \gamma_{\times2}^{t}(T''')\leq|S'''|=\gamma_{\times2}^{t}(T)-2k.$$
This completes the proof.
\end{proof}

In what follows we construct a family of directed trees in order to characterize those ones attaining the lower bound in the next theorem.

Let $F=F_{0}$ be a directed forest containing $r$ copies of the directed path $P_{2}$ with arcs $(v_{11},v_{12}),\cdots(v_{r1},v_{r2})$ and $r'$ copies of directed stars $H_{1},\cdots,H_{r'}$ of orders at least $3$ with central vertices $u_{1},\cdots,u_{r'}$, respectively. Let $q=r+r'-1$. We construct the sequence $F_{0},F_{1},\cdots,F_{q}$ of digraphs as follows. Let $F_{1}$ be obtained from $F_{0}$ by adding a vertex $w_{1}$ and two arcs $(x_{1},w_{1})$ and $(y_{1},w_{1})$ for some $x_{1},y_{1}\in A=\{v_{i1},v_{i2}\}_{i=1}^{r}\cup\{u_{i}\}_{i=1}^{r'}$ such that $x_{1}$ and $y_{1}$ are not vertices of a same $P_{2}$-copy. We now obtain $F_{2}$ from $F_{1}$ by adding a vertex $w_{2}$ and two arcs $(x_{2},w_{2})$ and $(y_{2},w_{2})$ such that $x_{2},y_{2}\in A$, $x_{2}\in N_{F_{0}}[\{x_{1},y_{1}\}]$ and $y_{2}\in V(F_{0})\setminus N_{F_{0}}[\{x_{1},y_{1}\}]$ (note that $\{x_{1},y_{1}\}=N^{-}_{F_{1}}(w_{1})$). We now suppose that $F_{j}$ is obtained from $F_{j-1}$ by adding a vertex $w_{j}$ with two arcs arcs $(x_{j},w_{j})$ and $(y_{j},w_{j})$ such that ($i$) $x_{j},y_{j}\in A$, and ($ii$) precisely one of them belongs to $B_{j}=N_{F_{0}}[N^{-}_{F_{j-1}}(\{w_{1},\cdots,w_{j-1}\})]$. We define $\Gamma$ as the family of all digraphs $F_{q}$ constructed as above. In what follows, we first need to show that the above construction is well-defined. Moreover, $\Gamma$ is a family of directed trees (Figure \ref{fig1} depicts a representative member of $\Gamma$).

\begin{proposition}\label{P1}
The following statements hold.\vspace{1mm}\\
\emph{(}a\emph{)} If $F_{0}$ is neither the directed path $P_{2}$ nor a directed star on at least three vertices, then there always exist two arcs $(x_{j},w_{j})$ and $(y_{j},w_{j})$ with the given properties \emph{(}i\emph{)} and \emph{(}ii\emph{)}, for all $1\leq j\leq q$.\vspace{1mm}\\
\emph{(}b\emph{)} $F_{q}$ is a directed tree.
\end{proposition}
\begin{proof}
($a$) Let $F_{0}$ be neither the directed path $P_{2}$ nor a directed star on at least three vertices. Therefore, $r+r'\geq2$. Clearly, there are such arcs for the vertex $w_{1}$. Let $2\leq i\leq q$ be the smallest index for which there is no a pair of arcs $(x_{i},w_{i})$ and $(y_{i},w_{i})$ with the properties ($i$) and ($ii$). Now, consider the digraph $F_{i-1}$. We then add the vertex $w_{i}$. Since $i\leq q=r+r'-1$, it follows that there exists a vertex $y_{i}\in A\setminus B_{i-1}$. Thus, we have the arcs $(x_{i},w_{i})$ and $(y_{i},w_{i})$ in which $x_{i}$ is a vertex in $B_{i-1}$. This contradicts our choice of $i$.

($b$) It is easy to see that $F_{1}$ is a directed forest. Assume now that $F_{i}$, $i\geq1$, is a directed forest. It follows from the way we construct $F_{i+1}$ from $F_{i}$ that the underlying graph of $F_{i+1}$ has no cycle as a subgraph. So, $F_{i+1}$ is a directed forest as well. In particular, $F_{q}$ is a directed forest. Now let $H_{1}\cdots,H_{r'}$ be those $r'$ directed stars in the definition of $F=F_{0}$ of order $t_{1},\cdots,t_{r'}\geq3$, respectively. Then, $|A(F_{q})|=3r+r'+t_{1}+\cdots+t_{r'}-2=|V(F_{q})|-1$. This implies that $F_{q}$ is a directed tree.
\end{proof}

We are now in a position to present the main theorem of this section.

\begin{theorem}\label{T3}
Let $T$ be a directed tree of order $n$ with $e$ end-vertices and $p$ penultimate vertices. Then,
$$\gamma^{t}_{\times2}(T)\geq\frac{2n+e-p+2}{3}.$$
The equality holds if and only if $T\in \Gamma$.
\end{theorem}
\begin{proof}
Let $S=\{v_{1},\cdots,v_{|S|}\}$ be a $\gamma^{t}_{\times2}(T)$-set. Note that all end-vertices and penultimate vertices belong to $S$, necessarily. Therefore, all pendant arcs belong to $A(T<S>)$. Since every vertex in $V(T)\setminus S$ is adjacent from at least two vertices in $S$ and $T<S>$ has no isolated vertices, it follows that
\begin{equation}\label{EQ2}
deg^{+}(v_{1})+\cdots+deg^{+}(v_{|S|})=|(S,V(T)\setminus S)_{T}|+|A(T<S>)|\geq2(n-|S|)+e+\frac{|S|-e-p}{2}.
\end{equation}
On the other hand,
\begin{equation}\label{EQ3}
deg^{+}(v_{1})+\cdots+deg^{+}(v_{|S|})\leq n-1.
\end{equation}
The desired lower bound now follows from (\ref{EQ2}) and (\ref{EQ3}).

Let $T\in \Gamma$. Let $S'$ be the set of vertices of the copies $P_{2}$ and $H_{i}$, $1\leq i\leq r'$. It is easy to see that $S'$ is a total $2$-dominating set in $T$ of cardinality $2r+t_{1}+\cdots+t_{r'}$, in which $t_{i}=|V(H_{i})|$ for $1\leq i\leq r'$. Moreover, $n=2r+t_{1}+\cdots+t_{r'}+r+r'-1$ and $e-p=t_{1}+\cdots+t_{r'}-2r'$. So, $\gamma^{t}_{\times2}(T)\leq 2r+t_{1}+\cdots+t_{r'}=(2n+e-p+2)/3$ which implies the equality in the lower bound.

Suppose now that we have the equality in the lower bound of the theorem. Then both the inequalities in (\ref{EQ2}) and (\ref{EQ3}) hold with equality, necessarily. In particular,
$$deg^{+}(v_{1})+\cdots+deg^{+}(v_{|S|})=n-1=\sum_{v\in V(T)}deg^{+}(v)$$
shows that $V(T)\setminus S$ is independent. Moreover, the equalities $|(S,V(T)\setminus S)_{T}|=2(n-|S|)$ and $|A(T\langle S\rangle)|=e+(|S|-e-p)/2$ imply that every vertex in $V(T)\setminus S$ is adjacent from precisely two vertices in $S$ and $T\langle S\rangle$ is a disjoint union of digraphs which are isomorphic to the directed paths $P_{2}$, and the directed stars $H_{t}$ of orders at least three whose end-vertices are not adjacent with the vertices in $V(T)\setminus S$. Choose a vertex $w_{1}\in V(T)\setminus S$. Then $w_{1}$ is adjacent from two vertices in $S$ which do not belong to a same component of $T'=T\langle S\rangle$, for otherwise the underlying graph of $T$ contains a cycle. Since $T$ is connected and its underlying graph has no cycle, there exists a vertex $w_{2}\in V(T)\setminus S$ which is adjacent from exactly one vertex in $N_{T'}[N^{-}(w_{1})]$. In general, by choosing the vertex $w_{j-1}$ in such a way, we find the vertex $w_{j}$ for which exactly one of its two in-neighbors belongs to $N_{T'}[N^{-}(\{w_{i}\}_{i< j})]$. The above argument shows that $T\in \Gamma$.
\end{proof}

\begin{figure}
  \centering
\begin{tikzpicture}[scale=.5, transform shape]
\node [draw, shape=circle] (v11) at  (-12,0) {};
\node [draw, shape=circle] (v12) at  (-10,0) {};
\node [draw, shape=circle] (v21) at  (-7,0) {};
\node [draw, shape=circle] (v22) at  (-5,0) {};
\node [draw, shape=circle] (v31) at  (-2,0) {};
\node [draw, shape=circle] (v32) at  (0,0) {};
\node [draw, shape=circle] (u1) at  (3,0) {};
\node [draw, shape=circle] (u2) at  (6,0) {};
\node [draw, shape=circle] (u11) at  (2,2) {};
\node [draw, shape=circle] (u12) at  (3,2) {};
\node [draw, shape=circle] (u13) at  (4,2) {};
\node [draw, shape=circle] (u21) at  (5.25,2) {};
\node [draw, shape=circle] (u22) at  (6.75,2) {};
\node [draw, shape=circle] (w1) at  (-8,-5) {};
\node [draw, shape=circle] (w2) at  (-5,-5) {};
\node [draw, shape=circle] (w3) at  (-2,-5) {};
\node [draw, shape=circle] (w4) at  (1,-5) {};

\draw[-Latex,line width=0.5pt](v11)--(v12);
\draw[-Latex,line width=0.5pt](v21)--(v22);
\draw[-Latex,line width=0.5pt](v31)--(v32);
\draw[-Latex,line width=0.5pt](u1)--(u11);
\draw[-Latex,line width=0.5pt](u12)--(u1);
\draw[-Latex,line width=0.5pt](u1)--(u13);
\draw[-Latex,line width=0.5pt](u21)--(u2);
\draw[-Latex,line width=0.5pt](u22)--(u2);
\draw[-Latex,line width=0.5pt](v11)--(w1);
\draw[-Latex,line width=0.5pt](v22)--(w1);
\draw[-Latex,line width=0.5pt](v12)--(w2);
\draw[-Latex,line width=0.5pt](u1)--(w2);
\draw[-Latex,line width=0.5pt](v21)--(w3);
\draw[-Latex,line width=0.5pt](v32)--(w3);
\draw[-Latex,line width=0.5pt](u2)--(w4);
\draw[-Latex,line width=0.5pt](v31)--(w4);

\end{tikzpicture}
  \caption{A member of $\Gamma$.}\label{fig1}
\end{figure}
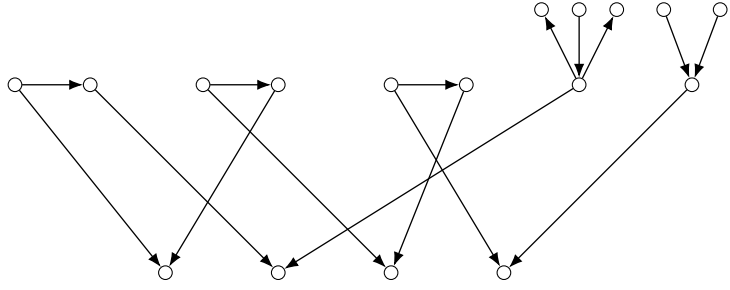


\section{Sum and product of $\Psi(D)$ and $\Psi(D^{-1})$ when $\Psi\in\{\gamma_{\times2}^{t},L_{2}^{t}\}$}

For the rest of the paper, we study the sum and product of the total $2$-domination number and the total $2$-limited packing number of a digraph and its converse. In order to obtain such inequalities concerninig the $2$-limited packing number, we make use of the structures of directed trees. Note that the study of these kinds of inequalities was first presented by Chartrand et. al \cite{chy} for the domination number. Since then bounds on $\Psi(D)+\Psi(D^{-1})$ or $\Psi(D)\Psi(D^{-1})$ appeared in literature, in which $\Psi$ is a digraph parameter. For example, the reader can be referred to the papers \cite{fur}, \cite{hq} and \cite{hqx}.

Nordhaus and Gaddum \cite{ng} in 1956, gave lower and upper bounds on the sum and product of the chromatic numbers of a graph $G$ and its complement $\overline{G}$ in terms of the order of $G$. Since then, bounds on $\Theta(G)+\Theta(\overline{G})$ or $\Theta(G)\Theta(\overline{G})$ are called {\em Nordhaus-Gaddum inequalities}, where $\Theta$ is any graph parameter. The search of the Nordhaus-Gaddum type inequalities has centered the attention of a large number of investigations, and in domination theory, this has probably been even more emphasized. For more information about this subject the reader can consult \cite{ah}.

Indeed, the above-mentioned inequalities concerning digraphs and their converse can be interpreted as modified Nordhaus-Gaddum theorems for digraphs, where the converse of a digraph replaces the complement of a graph.

We begin with these inequalities for the total $2$-domination number.

\begin{proposition}\label{NG1}
Let $D$ be a connected digraph of order $n\geq2$. Then, $\gamma^{t}_{\times2}(D)=n$ if and only if $deg^{-}(v)\leq1$ for each vertex $v$ which is neither a penultimate vertex nor an end vertex.
\end{proposition}
\begin{proof}
The sufficiency of the condition is clear. Now let $\gamma^{t}_{\times2}(D)=n$. Suppose that there exists a vertex $v$ with $deg^{-}(v)\geq2$ which is neither a penultimate vertex nor an end vertex. We can deduce that $V(D)\setminus\{v\}$ is a total $2$-dominating set in $D$, which is impossible.
\end{proof}

As an immediate consequence of Proposition \ref{NG1}, we have the following result.

\begin{corollary}
Let $D$ be a connected digraph of order $n\geq2$. Then, $\gamma^{t}_{\times2}(D)+\gamma^{t}_{\times2}(D^{-1})=2n$ \emph{(}$\gamma^{t}_{\times2}(D)\gamma^{t}_{\times2}(D^{-1})=n^{2}$\emph{)} if and only if $deg^{-}(v),deg^{+}(v)\leq1$ for each vertex $v$ which is neither a penultimate vertex nor an end vertex.
\end{corollary}

We now turn our attention to the total $2$-limited packing number. Let $D$ be a connected digraph of order $n$. Then, $L^{t}_{2}(D)+L^{t}_{2}(D^{-1})=2n$ when $n=1,2$. So, in what follows we may assume that $n\geq3$.

\begin{theorem}\label{T-NG2}
For any connected digraph $D$ of order $n\geq3$,
$$L^{t}_{2}(D)+L^{t}_{2}(D^{-1})\leq\frac{16n}{9}.$$
This bound is sharp.
\end{theorem}
\begin{proof}
We first prove that the inequality holds for all directed stars on $n\geq3$ vertices.\vspace{1mm}\\
\textit{Claim A.} \textit{For any directed star $S$ of order $n\geq3$, $L^{t}_{2}(S)+L^{t}_{2}(S^{-1})\leq16n/9$.}\vspace{1mm}\\
\textit{Proof of Claim A.} Let $u$ be the central vertex of $S$, $deg^{+}(u)=a$ and $deg^{-}(u)=b$. We have,
\begin{equation*}
L^{t}_{2}(S)+L^{t}_{2}(S^{-1})=\left \{
\begin{array}{lll}
a+b+2=n+1, & \mbox{if}\ a=0\ \mbox{or}\ b=0, \vspace{1.5mm}\\
a+b+2=n+1, & \mbox{if}\ a=b=1, \vspace{1.5mm}\\
a+b+3=n+2, & \mbox{if}\ min\{a,b\}=1\ \mbox{and}\ max\{a,b\}\geq2, \vspace{1.5mm}\\
a+b+4=n+3, & \mbox{if}\ a,b\geq2.
\end{array}
\right.
\end{equation*}
We now have $L^{t}_{2}(S)+L^{t}_{2}(S^{-1})\leq16n/9$ in all four possible values for $L^{t}_{2}(S)+L^{t}_{2}(S^{-1})$. ($\square$)

We are now able to extend the inequality in Claim A to directed trees on at least three vertices as follows. Indeed, we prove that
\begin{equation}\label{NGT}
L^{t}_{2}(T)+L^{t}_{2}(T^{-1})\leq16n/9,
\end{equation} by induction on the order $n\geq3$ of directed tree $T$. If $n=3$, then the result follows from Claim A. Suppose now that the result is true for all directed trees $T'$ of order $3\leq n'<n$. Let $T$ be a directed tree of order $n$. If $T$ is a directed star, then the result again follows by Claim A. So, we assume that $T$ is not a directed star. Therefore, $T$ has an arc $(x,y)$ such that $T-(x,y)$ is isomorphic to two non-trivial directed trees $T_{1}$ and $T_{2}$ of order $n_{1}<n$ and $n_{2}<n$, respectively. Moreover, by the symmetry between $T$ and $T^{-1}$, we may assume that $x\in V(T_{1})$ and $y\in V(T_{2})$. If $n_{1}=n_{2}=2$, then $T$ is obtained from orienting the edges of a path on four vertices. It follows that $L^{t}_{2}(T)+L^{t}_{2}(T^{-1})\leq6\leq 16(4)/9$. So, in what follows we may assume that $n_{1}\geq3$. We now distinguish two cases depending on $n_{2}$.

\textit{Case 1.} $n_{2}\geq3$. By the induction hypothesis, we have $L^{t}_{2}(T_{1})+L^{t}_{2}(T_{1}^{-1})\leq16n_{1}/9$ and $L^{t}_{2}(T_{2})+L^{t}_{2}(T_{2}^{-1})\leq16n_{2}/9$. The inequality (\ref{NGT}) now follows from the fact that $L^{t}_{2}(T)\leq L^{t}_{2}(T_{1})+L^{t}_{2}(T_{2})$ and $L^{t}_{2}(T^{-1})\leq L^{t}_{2}(T_{1}^{-1})+L^{t}_{2}(T_{2}^{-1})$.

\textit{Case 2.} $n_{2}=2$. Let $V(T_{2})=\{y,z\}$. We now consider two subcases.

\textit{Subcase 2.1.} Suppose that there exist a $L^{t}_{2}(T)$-set $B$ and a $L^{t}_{2}(T^{-1})$-set $B^{-1}$ such that $\{y,z\}\nsubseteq B\cap B^{-1}$. By the induction hypothesis, we have $L^{t}_{2}(T_{1})+L^{t}_{2}(T_{1}^{-1})\leq16n_{1}/9=16(n-2)/9$. Therefore,
$$L^{t}_{2}(T)+L^{t}_{2}(T^{-1})\leq L^{t}_{2}(T_{1})+|B\cap\{y,z\}|+L^{t}_{2}(T_{1}^{-1})+|B^{-1}\cap\{y,z\}|\leq \frac{16(n-2)}{9}+3<\frac{16n}{9}.$$

\textit{Subcase 2.2.} Suppose now that $\{y,z\}\subseteq B\cap B^{-1}$ for all $L^{t}_{2}(T)$-sets $B$ and $L^{t}_{2}(T^{-1})$-sets $B^{-1}$. This implies that $x\notin B\cup B^{-1}$. We have, $T_{1}=T-\{y,z\}$. Suppose that $T_{11},\cdots,T_{1p}$ are the components of $T_{1}-x$. If $|V(T_{11})|,\cdots,|V(T_{1p})|\geq3$, we have $L^{t}_{2}(T_{1i})+L^{t}_{2}(T_{1i}^{-1})\leq16|V(T_{1i})|/9$ for all $1\leq i\leq p$, by the induction hypothesis. Therefore,
\begin{equation*}
\begin{array}{lcl}
L^{t}_{2}(T)+L^{t}_{2}(T^{-1})&\leq& \sum_{i=1}^{p}(L^{t}_{2}(T_{1i})+L^{t}_{2}(T_{1i}^{-1}))+|B\cap\{x,y,z\}|+|B^{-1}\cap\{x,y,z\}|\\
&\leq& \frac{16}{9}\sum_{i=1}^{p}|V(T_{1i})|+4=\frac{16}{9}(n-3)+4<\frac{16n}{9}.
\end{array}
\end{equation*}
So, we assume that $|V(T_{1i})|\leq2$ for some $1\leq i\leq p$. Without loss of generality, we assume that $|V(T_{11})|,\cdots,|V(T_{1q})|\leq2$ for $1\leq q\leq p$. By the induction hypothesis, we have
\begin{equation*}\label{EQ4}
L^{t}_{2}(T_{1i})+L^{t}_{2}(T_{1i}^{-1})\leq16|V(T_{1i})|/9,
\end{equation*}
for all $q+1\leq i\leq p$ (if there is such an index $i$).

Let $x_{1},\cdots,x_{q}$ be the unique vertices in $V(T_{11}),\cdots,V(T_{1q})$, respectively, which are adjacent with $x$. We now present the following claim.\vspace{1mm}\\
\textit{Claim B.} \textit{If $q\geq4$, then at least one of the vertices $x_{1},\cdots,x_{q}$ does not belong to $B\cap B^{-1}$.}\vspace{1mm}\\
\textit{Proof of Claim B.} Suppose that $q\geq4$. Suppose to the contrary that $x_{1},\cdots,x_{q}\in B\cap B^{-1}$. Since $T$ is a directed tree, the subdigraph $H$ induced by $\{y,x,x_{1},\cdots,x_{4}\}$ is a directed star on six vertices. In fact, $H$ is isomorphic to one of the directed stars depicted in Figure \ref{fig0}. In the first three directed stars (from left to right) we have contradiction with the fact that $B$ is a $2$-limited packing in $D$, and in the last two directed stars we have contradiction with the fact that $B^{-1}$ is a $2$-limited packing in $D^{-1}$. Thus, at least one of the vertices $x_{1},\cdots,x_{q}$ does not belong to $B\cap B^{-1}$. ($\square$)

Let $q\geq4$. We now assume, without loss of generality, that $x_{1}\notin B\cap B^{-1}$. If $V(T_{11})=\{x_{1}\}$, then
\begin{equation*}\label{EQ5}
\begin{array}{lcl}
L^{t}_{2}(T)+L^{t}_{2}(T^{-1})&\leq&|B\cap\{y,z\}|+|B^{-1}\cap\{y,z\}|+|B\cap\{x_{1}\}|+|B^{-1}\cap\{x_{1}\}|\\
&+&|B\cap V(T-\{y,z,x_{1}\})|+|B^{-1}\cap V(T^{-1}-\{y,z,x_{1}\})|\\
&\leq& 4+1+L^{t}_{2}(T-\{y,z,x_{1}\})+L^{t}_{2}(T^{-1}-\{y,z,x_{1}\})\\
&\leq& 5+16(n-3)/9<16n/9.
\end{array}
\end{equation*}
If $V(T_{11})=\{x_{1},x'_{1}\}$ for some vertex $x'_{1}$, then
\begin{equation*}\label{EQ6}
\begin{array}{lcl}
L^{t}_{2}(T)+L^{t}_{2}(T^{-1})&\leq&4+|B\cap\{x_{1},x'_{1}\}|+|B^{-1}\cap\{x_{1},x'_{1}\}|+|B\cap V(T-\{y,z,x_{1},x'_{1}\})|\\
&+&|B^{-1}\cap V(T^{-1}-\{y,z,x_{1},x'_{1}\})|\\
&\leq& 4+3+L^{t}_{2}(T-\{y,z,x_{1},x'_{1}\})+L^{t}_{2}(T^{-1}-\{y,z,x_{1},x'_{1}\})\\
&\leq& 7+16(n-4)/9<16n/9.
\end{array}
\end{equation*}

It remains for us to prove the desired inequality when $q\leq3$. We consider two subcases depending on $p$ and $q$.

\textit{Subcase 2.2.1.} Let $p=q$. If $q=1$, then we deal with a directed tree whose underlying graph is a path on four or five vertices. Clearly, in such cases the desired inequality holds. If $q=2$ or $3$, then the directed tree $T$ is obtained from the orientation of a tree of order $n\in\{5,6,7\}$ or $n\in\{6,7,8,9\}$, respectively. In all the possible cases, we have $L^{t}_{2}(T)+L^{t}_{2}(T^{-1})\leq16n/9$.

\textit{Subcase 2.2.2. $q<p$.} Let $x_{p}$ be the unique vertex of $T_{1p}$ which is adjacent with $x$. Since both $T_{1p}$ and $T-V(T_{1p})$ have at least three vertices, it follows by the induction hypothesis that $L^{t}_{2}(T_{1p})+L^{t}_{2}(T_{1p}^{-1})\leq16|V(T_{1p})|/9$ and $L^{t}_{2}(T-V(T_{1p}))+L^{t}_{2}(T^{-1}-V(T_{1p}))\leq16(n-|V(T_{1p})|)/9$. Therefore,
\begin{equation*}\label{EQ7}
L^{t}_{2}(T)+L^{t}_{2}(T^{-1})\leq L^{t}_{2}(T_{1p})+L^{t}_{2}(T-V(T_{1p}))+L^{t}_{2}(T_{1p}^{-1})+L^{t}_{2}(T^{-1}-V(T_{1p}))\leq\frac{16n}{9}.
\end{equation*}
Indeed, in all possible cases we have proved the inequality (\ref{NGT}) for a directed tree $T$.

Since $D$ is connected, it has a spanning directed tree $T$. Moreover. $L^{t}_{2}(D)\leq L^{t}_{2}(T)$ and $L^{t}_{2}(D^{-1})\leq L^{t}_{2}(T^{-1})$. We now have
\begin{equation*}
L^{t}_{2}(D)+L^{t}_{2}(D^{-1})\leq L^{t}_{2}(T)+L^{t}_{2}(T^{-1})\leq\frac{16n}{9}
\end{equation*}
by (\ref{NGT}), as desired.

In what follows, we show that the upper bound is sharp. Let $D'$ be an arbitrary connected digraph on the set of vertices $V(D')=\{v_{1},\cdots,v_{n'}\}$. For every $1\leq i\leq n'$, we add four directed paths $P_{i1}:x^{i}_{11},x^{i}_{12}$, $P_{i2}:x^{i}_{21},x^{i}_{22}$, $P_{i3}:x^{i}_{31},x^{i}_{32}$ and $P_{i4}:x^{i}_{41},x^{i}_{42}$, and four arcs $(v_{i},x^{i}_{11}),(v_{i},x^{i}_{21}),(x^{i}_{31},v_{i}),(x^{i}_{41},v_{i})$. Let $R$ be the obtained digraph. It is easy to observe that $|V(R)|=9n'$ and $B=\{x^{i}_{11},x^{i}_{12},\cdots,x^{i}_{41},x^{i}_{42}\}_{i=1}^{n'}$ is both a $L^{t}_{2}(R)$-set and a $L^{t}_{2}(R^{-1})$-set. Thus, $L^{t}_{2}(R)+L^{t}_{2}(R^{-1})=16n'=16|V(R)|/9$. This completes the proof.

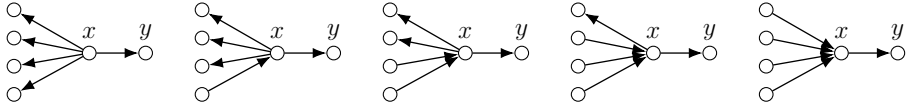
\begin{figure}
  \centering
\begin{tikzpicture}[scale=.5, transform shape]
\node [draw, shape=circle] (y) at  (-8.5,3.35) {};
\node [draw, shape=circle] (x) at  (-10,3.35) {};
\node [draw, shape=circle] (x1) at  (-12,4.5) {};
\node [draw, shape=circle] (x2) at  (-12,3.75) {};
\node [draw, shape=circle] (x3) at  (-12,3) {};
\node [draw, shape=circle] (x4) at  (-12,2.25) {};

\draw[-Latex,line width=0.5pt](x)--(y);
\draw[-Latex,line width=0.5pt](x)--(x1);
\draw[-Latex,line width=0.5pt](x)--(x2);
\draw[-Latex,line width=0.5pt](x)--(x3);
\draw[-Latex,line width=0.5pt](x)--(x4);

\node [scale=1.6] at (-10,3.9) {\large $x$};
\node [scale=1.6] at (-8.5,3.9) {\large $y$};
\node [draw, shape=circle] (b) at  (-3.5,3.35) {};
\node [draw, shape=circle] (a) at  (-5,3.35) {};
\node [draw, shape=circle] (a1) at  (-7,4.5) {};
\node [draw, shape=circle] (a2) at  (-7,3.75) {};
\node [draw, shape=circle] (a3) at  (-7,3) {};
\node [draw, shape=circle] (a4) at  (-7,2.25) {};

\draw[-Latex,line width=0.5pt](a)--(b);
\draw[-Latex,line width=0.5pt](a)--(a1);
\draw[-Latex,line width=0.5pt](a)--(a2);
\draw[-Latex,line width=0.5pt](a)--(a3);
\draw[-Latex,line width=0.5pt](a4)--(a);

\node [scale=1.6] at (-5,3.9) {\large $x$};
\node [scale=1.6] at (-3.5,3.9) {\large $y$};
\node [draw, shape=circle] (d) at  (1.5,3.35) {};
\node [draw, shape=circle] (c) at  (0,3.35) {};
\node [draw, shape=circle] (c1) at  (-2,4.5) {};
\node [draw, shape=circle] (c2) at  (-2,3.75) {};
\node [draw, shape=circle] (c3) at  (-2,3) {};
\node [draw, shape=circle] (c4) at  (-2,2.25) {};

\draw[-Latex,line width=0.5pt](c)--(d);
\draw[-Latex,line width=0.5pt](c)--(c1);
\draw[-Latex,line width=0.5pt](c)--(c2);
\draw[-Latex,line width=0.5pt](c3)--(c);
\draw[-Latex,line width=0.5pt](c4)--(c);

\node [scale=1.6] at (0,3.9) {\large $x$};
\node [scale=1.6] at (1.5,3.9) {\large $y$};
\node [draw, shape=circle] (f) at  (6.5,3.35) {};
\node [draw, shape=circle] (e) at  (5,3.35) {};
\node [draw, shape=circle] (e1) at  (3,4.5) {};
\node [draw, shape=circle] (e2) at  (3,3.75) {};
\node [draw, shape=circle] (e3) at  (3,3) {};
\node [draw, shape=circle] (e4) at  (3,2.25) {};

\draw[-Latex,line width=0.5pt](e)--(f);
\draw[-Latex,line width=0.5pt](e)--(e1);
\draw[-Latex,line width=0.5pt](e2)--(e);
\draw[-Latex,line width=0.5pt](e3)--(e);
\draw[-Latex,line width=0.5pt](e4)--(e);

\node [scale=1.6] at (5,3.9) {\large $x$};
\node [scale=1.6] at (6.5,3.9) {\large $y$};
\node [draw, shape=circle] (h) at  (11.5,3.35) {};
\node [draw, shape=circle] (g) at  (10,3.35) {};
\node [draw, shape=circle] (g1) at  (8,4.5) {};
\node [draw, shape=circle] (g2) at  (8,3.75) {};
\node [draw, shape=circle] (g3) at  (8,3) {};
\node [draw, shape=circle] (g4) at  (8,2.25) {};

\draw[-Latex,line width=0.5pt](g)--(h);
\draw[-Latex,line width=0.5pt](g1)--(g);
\draw[-Latex,line width=0.5pt](g2)--(g);
\draw[-Latex,line width=0.5pt](g3)--(g);
\draw[-Latex,line width=0.5pt](g4)--(g);

\node [scale=1.6] at (10,3.9) {\large $x$};
\node [scale=1.6] at (11.5,3.9) {\large $y$};

\end{tikzpicture}
\caption{All the possible directed stars on six vertices with the center $x$ and fixed arc $(x,y)$.}\label{fig0}
\end{figure}
\end{proof}

Maximizing $L^{t}_{2}(D)L^{t}_{2}(D^{-1})$ subject to $L^{t}_{2}(D)+L^{t}_{2}(D^{-1})=16n/9$, we have $L^{t}_{2}(D)=L^{t}_{2}(D^{-1})=8n/9$. Therefore, we have the following upper bound for the product of $L^{t}_{2}(D)$ and $L^{t}_{2}(D^{-1})$. Furthermore, the bound is sharp for the digraph $R$ defined in the proof of Theorem \ref{T-NG2}.

\begin{corollary}
For any connected digraph $D$ of order $n\geq3$, $L^{t}_{2}(D)L^{t}_{2}(D^{-1})\leq64n^{2}/81$.
\end{corollary}


\section{Concluding remarks}

Given the integers $k,t\in\{1,2\}$, we proved that $L_{k}(D)\leq \frac{k}{t}\gamma_{\times t}(D)$ for all digraphs $D$ with $\delta^{-}(D)\geq t-1$. Note that the digraph parameters $L_{2}^{t}$ and $\gamma_{\times2}^{t}$ are the same for both the complete biorientation $\overleftrightarrow{K_{n}}$ of the complete graph $K_{n}$ of order $n\geq2$ and the digraph $R$ introduced in the proof of Theorem \ref{T-NG2}. Moreover, we proved that $L_{2}^{t}(T)\leq \gamma_{\times2}^{t}(T)$ for all nontrivial directed tree $T$. So, it is natural to present the following open problem.\vspace{1mm}\\
\textbf{Problem $1$.} Does the inequality $L_{2}^{t}(D)\leq \gamma_{\times2}^{t}(D)$ hold for any digraph $D$ with no isolated vertices?\vspace{1mm}

Mojdeh et al. \cite{msy} proved that $\rho(T)=\gamma(T)$, for all directed tree $T$. Although such a result does not hold for $L^{t}_{2}(T)$ and $\gamma^{t}_{\times2}(T)$, one can ask for the family of all directed trees for which these two parameters are the same. Indeed, we pose the following problem.\vspace{1mm}\\
\textbf{Problem $2$.} Characterize the directed trees $T$ for which $L^{t}_{2}(T)=\gamma^{t}_{\times2}(T)$.\\

\textbf{Acknowledgements}
This research work has been supported by a research grant from the University of Mazandaran.


\begin{thebibliography}{}

\bibitem {ah} M. Aouchiche and P. Hansen, {\em A survey of Nordhaus-Gaddum type relations}, Discrete Appl. Math. {\bf 161} (2013), 466--546.
\bibitem {ajv} S. Arumugam, K. Jacob and L. Volkmann, {\em Total and connected domination in digraphs}, Australas. J. Combin. {\bf 39} (2007), 283--292.
\bibitem {bbg} P.N. Balister, B. Bollobas and K. Gunderson, {\em Limited packings of closed neighbourhoods in graphs}, arXiv: 1501.01833v1 [math.CO] 8 Jan 2015.
\bibitem {bg} J. Bang-Jensen and G. Gutin, Digraphs: Theory, Algorithms and Applications, Springer Monographs in Mathematics, Springer-Verlag London Ltd. London (2007).
\bibitem {b} C. Berge, Theory of Graphs and its Applications, Methuen, London, 1962.
\bibitem {chy} G. Chartrand, F. Harary and B. Quan Yue, {\em On the out-domination and in-domination numbers of a digraph}, Discrete Math. {\bf 197}/{\bf 198} (1999), 179--183.
\bibitem {cg} N.E. Clarke and R.P. Gallant, {\em On $2$-limited packings of complete grid graphs}, Discrete Math. {\bf 340} (2017), 1705--1715.
\bibitem {fj1} J.F. Fink and M.S. Jacobson, {\em $n$-domination in graphs}, in: Graph Theory with Applications to Algorithms and Computer Science, John Wiley and Sons, (1985), 282--300.
\bibitem {fj2} J.F. Fink and M.S. Jacobson, {\em On $n$-domination, $n$-dependence and forbidden subgraphs}, in: Graph Theory with Applications to Algorithms and Computer Science, John Wiley and Sons, New York, (1985), 301--311.
\bibitem {fu} Y. Fu, {\em Dominating set and converse dominating set of a directed graph}, Am. Math. Mon. {\bf 75} (1968), 861--863.
\bibitem {fur} M. Furuya, {\em Bounds on the domination number of a digraph and its reverse}, Filomat, {\bf 32} (2018), 2517--2524.
\bibitem {gghr} R. Gallant, G. Gunther, B.L. Hartnell and D.F. Rall, {\em Limited packing in graphs}, Discrete Appl. Math. {\bf 158} (2010), 1357--1364.
\bibitem {packing} D.S. Hochbaum and D.B. Schmoys, {\em A best possible heuristic for the $k$-center problem}, Math. Oper. Res. {\bf 10} (1985), 180--184.
\bibitem {hq} G. Hao and J. Qian, {\em On the sum of out-domination number and in-domination number of digraphs}, Ars Combin. {\bf 119} (2015), 331--337.
\bibitem {hqx} G. Hao, J. Qian and Z. Xie, {\em On the sum of the total domination numbers of a digraph and its converse}, Quaest. Math. {\bf 42} (2019), 47--57.
\bibitem {h} F. Harary, {\em The Number of Functional Digraphs}, Math. Annalen, {\bf 138} (1959), 203--210.
\bibitem {hh} F. Harary and T.W. Haynes, {\em Double domination in graphs}, Ars Combin. {\bf 55} (2000), 201--213.
\bibitem {hnc} F. Harary, R.Z. Norman and D. Cartwright, Structural Models, Wiley, New York, 1965.
\bibitem {hhs2} T.W. Haynes, S.T. Hedetniemi and P.J. Slater, Fundamentals of Domination in Graphs, Marcel Dekker, New York, 1998.
\bibitem{l1} C. Lee, {\em Domination in digraphs}, J. Korean Math. Soc. {\bf 35} (1998), 843--853.
\bibitem {msy} D.A. Mojdeh, B. Samadi and I.G. Yero, {\em Packing and domination parameters in digraphs}, Discrete Appl. Math., doi: 10.1016/j.dam.2019.04.008.
\bibitem {ng} E.A. Nordhaus and J. Gaddum, {\em On complementary graphs}, Amer. Math. Monthly, {\bf 63} (1956), 175--177.
\bibitem {O} O. Ore, Theory of Graphs, Amer. Math. Soc. Colloq. Publ., {\bf 38} (Amer. Math. Soc., Providence, RI), 1962.
\bibitem {obb} L. Ouldrabah, M. Blidia and A. Bouchou, {\em On the $k$-domination number of digraphs}, J. Combin. Optim., doi: 10.1007/s10878-019-00405-1.
\bibitem {we} D.B. West, Introduction to Graph Theory (Second Edition), Prentice Hall, USA, 2001.

\end{thebibliography}
\end{document}